\documentclass{article}

\usepackage{lipsum}
\usepackage{graphicx}
\usepackage{hyperref}

\usepackage{authblk}

\usepackage{tikz}
\usetikzlibrary{arrows,calc}

\tikzstyle{arc}=[->,shorten <=3pt, shorten >=3pt,
                 >=stealth, line width=1.1pt]
\tikzstyle{edge}=[shorten <=2pt, shorten >=2pt,
                  >=stealth, line width=1.1pt]
\tikzstyle{vertex}=[circle, fill=white, draw,
                    minimum size=5pt,
                    inner sep=0pt, outer sep=0pt]

\usepackage{float}

\usepackage{xcolor}

\usepackage{amsmath}
\usepackage{amssymb}
\usepackage{dsfont}
\usepackage{amsfonts}

\usepackage{caption}
\usepackage{subcaption}
\usepackage{multirow}
\usepackage{tabularx}

\usepackage{amsthm}

\usepackage[capitalize]{cleveref}

\newtheorem{theorem}{Theorem}
\newtheorem{lemma}[theorem]{Lemma}
\newtheorem{corollary}[theorem]{Corollary}
\newtheorem{proposition}[theorem]{Proposition}
\newtheorem{conjecture}[theorem]{Conjecture}

\title{Long-eared digraphs%
\thanks{The authors gratefully acknowledge support from grants CONACYT FORDECYT-PRONACES CF-2019/39570 and UNAM-DGAPA-PAPIIT IN110724, SEP-CONACYT A1-S-8397 and DGAPA-PAPIIT IA101423. The first author is supported by postdoctoral grant “Estancias Posdoctorales por México” by CONACYT (CVU: 622815)}}

\author[1]{Germ\'an Ben\'itez-Bobadilla\thanks{german@ciencias.unam.mx}}
\author[2]{Hortensia Galeana-S\'anchez \thanks{hgaleana@matem.unam.mx}}
\author[1]{C\'esar Hern\'andez-Cruz \thanks{chc@ciencias.unam.mx}}

\affil[1]{Departamento de Matem\'aticas, Facultad de Ciencias, Universidad Nacional Aut\'onoma de M\'exico}
\affil[2]{Instituto de Matem\'aticas, Universidad Nacional Aut\'onoma de M\'exico}

\begin{document}
\date{}

\maketitle
\begin{abstract}
Let $H$ be a subdigraph of a digraph $D$. An ear of $H$ in $D$ is a path or a cycle in $D$ whose ends lie in $H$ but whose internal vertices do not. An \emph{ear decomposition} of a strong digraph $D$ is a nested sequence $(D_0,D_1,\ldots , D_k)$ of strong subdigraphs of $D$ such that: 1) $D_0$ is a cycle, 2) $D_{i+1} = D_i\cup P_i$, where $P_i$ is an ear of $D_i$ in $D$, for every $i\in \{0,1,\ldots,k-1\}$, and 3) $D_k=D$.

In this work, the $\mathcal{LE}_i$ is defined as the family of strong digraphs, with an ear decomposition such that every ear has a length of at least $i\geq 1$. It is proved that Seymour's second Neighborhood Conjecture and the Laborde, Payan, and Soung conjecture, are true in the family $\mathcal{LE}_2$, and the Small quasi-kernel conjecture is true for digraphs in $\mathcal{LE}_3$. Also, some sufficient conditions for a strong nonseparable digraph in $\mathcal{LE}_2$ with a kernel to imply that the previous (following) subdigraph in the ear decomposition has a kernel too, are presented. It is proved that digraphs in $\mathcal{LE}_2$ have a chromatic number at most 3, and a dichromatic number 2 or 3. Finally, the oriented chromatic number of asymmetrical digraphs in $\mathcal{LE}_3$ is bounded by 6, and it is shown that the oriented chromatic number of asymmetrical digraphs in $\mathcal{LE}_2$ is not bounded.

\end{abstract}

\section{Introduction}
In this work, we will consider finite digraphs with neither multiple arcs nor loops. Throughout this work, $D$ denotes a digraph; $V(D)$ and $A(D)$ are the set of vertices and arcs, respectively.
 For general concepts, we refer the reader to \cite{bang2009,classesofdigraphs2018book, bondy2008}. An arc $(u,v)$ of $D$ is asymmetrical (symmetrical) if and only if $(v,u)\in A(D)$ ($(v,u)\notin A(D)$). A digraph $D$ is \emph{asymmetrical} (symmetrical) if and only if every arc of $D$ is asymmetrical (symmetrical).
Unless otherwise stated, we only consider directed walks, directed paths, and directed cycles. 
A $k$-cycle ($k$-path) is a cycle (path) with $k$ vertices.
 The asymmetrical directed cycle (path) with order $n$ will be denoted by $\overrightarrow{C}_n$ ($\overrightarrow{P}_n$). A sink is a vertex with an out-degree equal to zero.
 
 Let $D$ be a digraph. The underlying graph of $D$ is the graph $UG(D)$ with the same vertex set as $D$ and $xy\in E(UG(D))$ if and only if $(x,y)\in A(D)$ or $(y,x)\in A(D)$. We say that $D$ is a nonseparable (connected) digraph if and only if $UG(D)$ is a nonseparable (connected) graph.
 
 Let $S$ be a subset of vertices of a digraph $D$. We say that $S$ is an independent set of $D$ if and only if for any pair of different vertices of $S$ there is no arc between them, and $S$ is absorbent in $D$ if and only if for every vertex $u\in V(D)\setminus S$ there is a vertex $v\in S$ such that $(u,v)\in A(D)$. A subset $N$ of vertices of $D$ is a \emph{kernel} of $D$ if and only if it is both absorbent and independent.
In the context of game theory, von Neumann and Morgenstern introduced the kernels, in the study of winning strategies in two-person games \cite{neu}. However, kernels have been widely investigated both for their theoretical interest and for their applications. In particular, Chvátal proved that the problem of determining whether a digraph has a kernel or not is NP-complete \cite{Chv73}, furthermore Hell and Hernández-Cruz proved that the problem remains NP-complete even when the underlying graph is 3-colorable \cite{Hell2014DMGT}.

A subset $S$ of $V(D)$ is quasi-absorbent if and only if for every vertex $x\in V(D)\setminus S$ there is a path from $x$ to some vertex in $S$ with length at most 2. We say that $S$ is a quasi-kernel of $D$ if and only if $S$ is an independent quasi-absorbent set. Unlike kernels, Chvátal and Lovász proved that every digraph has a quasi-kernel \cite{chvatal1974every}; and its properties have been studied ever since \cite{AI2023113435, erdos2010two, JACOB1996279, Kostochka2022}. However, in \cite{CROITORU2015863} it is proved that the decision problem of determining whether there is a quasi-kernel containing a specified vertex in a given digraph is NP-complete.

One of the most studied topics is coloring in graphs and digraphs. Among them are proper colorings. A $k$-coloring of $D$ is an assignment of colors to the vertices of $D$; that is, it is a function $c: V(D)\rightarrow \{1, \ldots, k\}$. We say that a $k$-coloring is proper if and only if adjacent vertices have different colors. The \emph{chromatic number} of $D$, $\chi(D)$ is the minimum $k$ such that there is a proper $k$-coloring of $D$. Equivalently a proper $k$-coloring of $D$ is a partition of $V(D)$ into $k$ independent subsets, called chromatic classes.

Let $D$ and $H$ be two digraphs. A \emph{homomorphism} from a $D$ to $H$ is function $\phi: V(D)\rightarrow V(H)$ such that if $(u,v)\in A(D)$, then $(\phi(u),\phi(v) )\in A(H)$. Note that $k$-coloring can be equivalently regarded as a homomorphism of $D$ to the complete symmetrical digraph $K_k$ on $k$-vertices. Hence, the chromatic number of $D$ also can be defined as the smallest $k$ such that $D$ there is a homomorphism of $D$ to $K_k$ and there is no homomorphism of $D$ to $K_{k-1}$.

Let $D$ be an asymmetrical digraph. An \emph{oriented $k$-coloring} of $D$ is a proper coloring such that all the arcs linking two color classes have the same direction.
Note that any two vertices linked by a path of length 1 or 2, must be assigned distinct colors in any oriented coloring.
The \emph{oriented chromatic number} of $D$, $\chi_o(D)$ is the minimum $k$ such that there is an oriented $k$-coloring of $D$; or, equivalently, is defined as the minimum order of tournament $T$ such that there is a homomorphism of $D$ to $T$.

The oriented chromatic number was introduced by Courcelle in \cite{COURCELLE1994117}, and has been studied by different authors, in particular, 
in \cite{KLOSTERMEYER2004161}, Klostermeyer and MacGillivray proved that for a given asymmetrical digraph $D$, whether $\chi_o(G)\leq k$ can be decided in polynomial time if $k\leq 3$ and it is NP-complete if $k\geq 4$, even when $D$ is connected.
On the other hand, Culus and Demange proved that deciding whether $\chi_o(G)\leq 4$ is NP-complete even if $D$ is a bounded degree bipartite asymmetrical digraph, or a bounded degree acyclic asymmetrical digraph \cite{Culus2006}.
 Even more, Coelho, Faria, Gravier, and Klein proved that this problem is still NP-complete when restricted to acyclic asymmetrical digraphs whose underlying graph is connected, planar, bipartite, and has a maximum degree 3 \cite{COELHO2016109}.

We say that a digraph $D$ is strong if and only if for every $x,y\in V(D)$ there are a path from $x$ to $y$ and a path from $y$ to $x$. Note that, every strong digraph, other than $K_1$, contains a cycle. Let $H$ be a subdigraph of a strong digraph $D$. A directed ear of $H$ in $D$ is a path in $D$ whose ends lie in $H$ but whose internal vertices do not. In \cite{bondy2008}, it is proved that if $H$ is a nontrivial proper nonseparable strong subdigraph of
a nonseparable strong digraph $D$, then $H$ has a directed ear in $D$. Also, it is shown that if $H$ is a strong subdigraph of a digraph $D$, and $P$ is an ear of $H$ in $D$, then $H\cup P$ is strong.

An \emph{ear decomposition} of a nonseparable strong digraph $D$ is a nested sequence $(D_0,D_1,\ldots , D_k)$ of nonseparable strong subdigraphs of $D$ such that:
\begin{enumerate}
\item $D_0$ is a cycle,
 \item $D_{i+1} = D_i\cup P_i$, where $P_i$ is an ear of $D_i$ in $D$, for every $i\in \{0,1,\ldots,k-1\}$, and
 \item $D_k=D$.
\end{enumerate}

Also, it is shown in \cite{bondy2008} that every nonseparable strong digraph, other than $K_1$, has an ear decomposition. It is worth noting that if the ears also be cycles, where only the vertex at which the cycle starts and ends belongs to the strong subdigraph, are allowed, then we obtain a strong subdigraph. Hence, a strongly connected digraph has an ear decomposition, even if it is not a nonseparable digraph. Therefore, unless we indicate otherwise, we assume that ears can also be cycles.

Let $i$ be a positive integer. We define $\mathcal{LE}_i$ as the family of strong digraphs, with an ear decomposition such that every ear has a length of at least $i$. By definition of $\mathcal{LE}_i$, we have that $\mathcal{LE}_{i+1}\subset \mathcal{LE}_i$. 
 Observe that $\mathcal{LE}_1$ is the family of strong digraphs, and $\mathcal{LE}_2$ is the family of strong digraphs, with an ear decomposition such that every ear has a length of at least two.
Note that if $D\in \mathcal{LE}_i$, for some positive integer $i\geq 2$, then $D$ is not a minor-closed digraph. Thus $\mathcal{LE}_i$ cannot be characterized by a family of forbidden minors.
By definition of $\mathcal{LE}_i$ it is not difficult to note that, for every positive $i$, and every $D$ in $\mathcal{LE}_i$, there is no $k\in \mathds{N}$ such that $\Delta (D)\leq k$, nor the girth of $D$ is at most $k$, nor the circumference of $D$ is at most $k$. Also, if $D\in \mathcal{LE}_i$, then $D$ is not necessarily Eulerian. Now, consider $D$ the digraph in \ref{fig nonplanar}, and let $D_0$, $D_1$, $D_2$ and $D_3$ be the subdigraphs of $D$ induced by $\{x_1,y_1,x_2,y_2\}$, $\{x_1,y_1,x_2,y_2,x_3,y_3\}$, $\{x_1,y_1,x_2,y_2,x_3,y_3,w_1,w_2\}$ and $V(D)$ respectively. It follows that $(D_0, D_1, D_2, D_3)$ is an ear decomposition of $D$ such that every ear has a length of at least two. It follows that $D\in \mathcal{LE}_2$. On the other hand, $UG(D)$ is a subdivision of $K_{3,3}$, by Kuratowski's Theorem, $UG(D)$ is not planar. Hence, if $D\in \mathcal{LE}_2$, then $D$ is not necessarily planar. 
Similarly, for each $i\geq 3$, digraphs in $\mathcal{LE}_i$ such that they are not planar, can be constructed.
From the above, we can conclude that $\mathcal{LE}_i$ is not properly contained in common families of digraphs, such as planar, Eulerian, bounded girth, bounded degree, or bounded circumference, for every $i\geq 1$.

\begin{figure}
\begin{center}
\includegraphics[scale=.7]{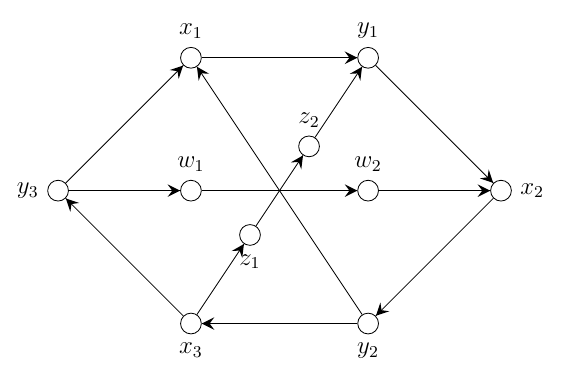}
\end{center}
\caption{Nonplanar digraph in $\mathcal{LE}_2$}
\label{fig nonplanar}
\end{figure}

The rest of the paper is organized as follows. \Cref{sec conjectures} is devoted to proving that two open problems in digraphs can be solved easily in the $\mathcal{LE}_2$ family, and a third problem can be easily solved in $\mathcal{LE}_3$ but for $\mathcal{LE}_2$ it remains an open problem. In \Cref{sec kernels}, some sufficient conditions for a strong nonseparable digraph in $\mathcal{LE}_2$ with a kernel, to imply that the previous subdigraph in the ear decomposition has a kernel, are presented. Also, some sufficient conditions for a strong nonseparable digraph in $\mathcal{LE}_2$, with a kernel to imply that the following digraph in an ear decomposition has a kernel too, are showed. In \Cref{sec classical colorings} it is proved that digraphs in $\mathcal{LE}_2$ have a chromatic number at most 3, and as a consequence, a digraph in $\mathcal{LE}_2$ has a dichromatic number 2 or 3. Finally, in \Cref{sec oriented coloring} the oriented chromatic number of asymmetrical digraphs in $\mathcal{LE}_3$ is bounded by 6. Furthermore, a family of digraphs with arbitrarily high oriented chromatic number where every ear has a length of 2 is presented, and we conclude that
the hypothesis that the ears have a length of at least 3 is tight.

\section{Some conjectures}\label{sec conjectures}

In this section, we will prove that some open problems in digraphs can be solved easily in the $\mathcal{LE}_i$ family, for some $i\geq 2$.

Let $D$ be a digraph, and let $v$ be a vertex of $D$.
We denote the second out-neighborhood of $v$ as 
$N^{++}(v)= N^+(N^+(v))= \underset{u\in N^+(v)}{\bigcup} N^+(u)\setminus N^+(v)$
 and $\delta^{++} (v) = |N^{++}(v)|$.
In 1990, Seymour proposed the following conjecture, which has been widely studied in different digraph families \cite{DAAMOUCH2020454, CGH2016, dean1995squaring, FY2007JGT, DAAMOUCH2021332}.

\begin{conjecture}\label{conj: SSNC}[Seymour's Second Neighborhood Conjecture \cite{dean1995squaring}]
If $D$ is an oriented graph, then there is a vertex $v$ such that $\delta^{++}D(v)\geq \delta^+ D(v)$, such $v$ is referred to as a Seymour vertex.
\end{conjecture}

Let $D$ be a digraph in $\mathcal{LE}_2$, with $(D_0, D_1,\ldots,D_k)$ an ear decomposition such that every ear has a length of at least 2. Consider $D_{k-1}$ and its ear $P=(x_0,x_1,\ldots,x_r)$ in $D$, with $r\geq 2$. Observe that $\delta^+(x_{r-1})=1$, and $\delta^+(x_r)\geq 1$. It follows that $x_{r-1}$ is a Seymour vertex in $D$. Therefore, we have the following result.

\begin{theorem}
The Seymour's Second Neighborhood Conjecture is true for every $D$ in $\mathcal{LE}_i$, for every $i\geq 2$.
\end{theorem}

In 1983, Laborde, Payan, and Xuong proposed the following conjecture:

\begin{conjecture}[\cite{LabordePayanXuong82}]\label{conj: ind longest path} Every digraph has an independent set meeting every longest path.
\end{conjecture}

Like the previous one, \Cref{conj: ind longest path} has been extensively studied by different authors, who have found sufficient conditions for the conjecture to hold, for example in \cite{LabordePayanXuong82, BANGJENSEN1993267, GALEANASANCHEZ20082460, GALEANASANCHEZ1996141, Havet2004}.

\begin{theorem}
Let $D$ be a strong 
digraph. If $D\in \mathcal{LE}_i$,
 then $D$ has an independent set meeting every longest path of $D$, for every $i\geq 2$.
\end{theorem}

\begin{proof}
Let $D$ be a strong digraph with an ear decomposition $(D_0,D_1,\ldots,D_k)$ such that every ear of $D_i$ has a length of at least two. 
We will proceed by induction over $k$. 
If $k=0$, then $D$ is a cycle. It follows that every longest path of $D$ is a Hamiltonian path, moreover every independent set of $D$ meets every longest path of $D$. Suppose that if $D'$ is a strong digraph with an ear decomposition $(D'_0,D'_1,\ldots,D'_k)$ such that every ear of $D'_i$ has a length of at least two, then $D'$ has an independent set meeting every longest path of $D'$.
Suppose that $D$ is a strong digraph with an ear decomposition $(D_0,D_1,\ldots,D_k,D_{k+1})$ such that every ear of $D_i$ has a length of at least two. Let $P_{k}=(x_0,x_1,\ldots,x_r)$ be the ear of $D_k$ in $D$.
 On the other hand, note that $D'=D_k$ is a strong digraph with an ear decomposition $(D_0,D_1,\ldots,D_k,)$ such that every ear of $D_i$ has a length of at least two, by induction hypothesis, $D'$ has an independent set $S$ meeting every longest path of $D'$.
  If there is no longest path in $D$, with some arc in $P_k$, then $S$ is also an independent set meeting every longest path of $D$.
  Assume that, at least, there is a longest path of $D$ with arcs in $P_k$.
Let $P$ be a longest path in $D$. If $P$ is also a longest path in $D'$, then $S$ meets $P$. Otherwise, $P$ has arcs of $P_k$, furthermore, by definition of $P$, at least one $(x_0,x_1,x_2,\ldots, x_{r-1})$ or $(x_1,x_2,\ldots, x_r)$ is a subpath of $P$. If $x_0$ and $x_r$ belong to $S$, then $S$ intersects $P$; even mover, if $x_0,x_r\notin S$, then $S\cup \{x_1\}$. intersects $P$. We consider the following cases.
  
  \emph{Case 1}. $x_0\in S$, $x_r\notin S$, and $l(P_k)\geq 3$. It follows that $S\cup \{x_2\}$ meets $P$.
  
    \emph{Case 2}. $x_0\in S$, $x_r\notin S$, and $l(P_k)=2$. If $x_0\notin V(P)$, then $x_1$ is the first vertex in $P$ and $(x_1,x_2=x_r)$ is a subpath of $P$, moreover, $(x_0,x_1)\cup P$ is a path with length greater than $P$ in $D$, which is impossible. Hence, $S$ meets $P$. 
    
    \emph{Case 3}. $x_0\notin S$, $x_r\in S$, and $l(P_k)\geq 3$. It follows that $S\cup \{x_1\}$ meets $P$.
  
    \emph{Case 4}. $x_0\notin S$, $x_r\in S$, and $l(P_k)=2$. If $x_r\notin V(P)$, then $x_1$ is the last vertex in $P$ and $(x_0,x_1)$ is a subpath of $P$, moreover, $P\cup (x_1,x_2)$ is a path with a length greater than $P$ in $D$, which is impossible. Hence, $S$ meets $P$.
   
Therefore, there is an independent set meeting every longest path in $D$.
\end{proof}

Let $D$ be a digraph. A subset of vertices $Q$ is \emph{quasi-absorbent} if and only if for every $x\notin Q$ there is a vertex $y\in Q$ such that $(x,y)\in A(D)$ or there are $w\in V(D)$ and $y\in Q$ such that $(x,w,y)$ is a path in $D$.
A \emph{quasi-kernel} of $D$ is an independent and quasi-absorbent set of $D$. A quasi-kernel $Q$ of $D$ is small if $|Q|\leq \frac{|V(D)|}{2}$. In \cite{chvatal74quasi}, Chv\'atal proved that every digraph has a quasi-kernel. However, the problem of deciding the existence of a quasi-kernel containing a specific vertex is an NP-complete problem \cite{CROITORU2015863}.
On the other hand, in 1976, Erd\H{o}s and Székely conjectured that every  digraph $D$ without sinks has a small quasi-kernel. Also, this conjecture has been studied by several authors \cite{AI2023113435, Gutin200448, HJ2008}. 

\begin{conjecture}\label{conj: small quasi}[Small quasi-kernel \cite{erdos2010two}]
Every digraph $D$ without a sink has a small quasi-kernel.
\end{conjecture}

Recently, Kostochka, Luo, and Shan proved that the conjecture is true when the digraph is an orientation of a 4-colorable graph \cite{Kostochka2022}, and van Hulst proved that digraphs with kernels satisfy the conjecture \cite{van2021kernels}.

\begin{theorem}\label{thm quasi LE3}
Let $D$ be a strong 
 digraph. If $D \in \mathcal{LE}_i$,
 then $D$ has a small quasi-kernel, for every $i\geq 3$.
\end{theorem}

\begin{proof}
Let $D$ be a strong digraph with an ear decomposition $(D_0, D_1, \ldots, D_k)$ such that every ear of $D_i$ has a length of at least three. 
We will proceed by induction over $k$. 
If $k=0$, then $D$ is a cycle. It follows that $D$ has a small quasi-kernel.
Suppose that if $D'$ is a strong digraph with an ear decomposition $(D'_0, D'_1,\ldots, D'_k)$ such that every ear of $D'_i$ has a length of at least three, then $D'$ has a small quasi-kernel.
Suppose that $D$ is a strong digraph with an ear decomposition $(D_0, D_1,\ldots, D_k, D_{k+1})$ such that every ear of $D_i$ has a length of at least three. Let $P_{k}=(x_0,x_1,\ldots,x_r)$ be the ear of $D_k$ in $D$.

Observe that $D'=D_k$ is a strong digraph with an ear decomposition $(D_0, D_1,\ldots, D_k,)$ such that every ear of $D_i$ has a length of at least three, by the induction hypothesis, $D'$ has a small quasi-kernel $Q$. Observe that $Q$ is an independent set in $D$ and a quasi-absorbent set in $D'$. 
 We have cases, depending on whether the extreme vertices of $P_k$ belong to $Q$ or not.
We construct the set $Q'$, following \Cref{tabla quasi}.

\begin{table}[!ht]
    \centering
    \begin{tabular}{|c|c|c|c|}
    \hline
\multicolumn{2}{|c|}{Cases} & length & $Q'$\\ \hline
      \multirow{3}{4em}{$x_0\in Q$}   &  \multirow{3}{4em}{$x_r\in Q$} & $r\equiv 0$ (mod 3) & $Q\cup \{ x_3,x_6,\ldots,x_{r-3}\}$  \\ \cline{3-4}
        &  & $r\equiv 1$ (mod 3) & $Q\cup \{x_2\}\cup \{ x_4,x_7,\ldots,x_{r-3}\}$  \\ \cline{3-4}
         &  & $r\equiv 2$ (mod 3) & $Q\cup \{x_2, x_5,\ldots,x_{r-3}\}$  \\ \hline
      \multirow{3}{4em}{$x_0\notin Q$}  & \multirow{3}{4em}{$x_r\in Q$}  & $r\equiv 0$ (mod 3) &$Q\cup \{ x_3,x_6,\ldots,x_{r-3}\}$ \\ \cline{3-4}
         &  & $r\equiv 1$ (mod 3) & $Q\cup \{x_1, x_4,\ldots,x_{r-3}\}$  \\ \cline{3-4}
         &  & $r\equiv 2$ (mod 3) & $Q\cup \{x_2, x_5,\ldots,x_{r-3}\}$  \\ \hline  
      \multirow{3}{4em}{$x_0\in Q$} & \multirow{3}{4em}{$x_r\notin Q$} & $r\equiv 0$ (mod 3) & $Q\cup \{ x_1,x_4,\ldots,x_{r-1}\}$  \\ \cline{3-4}
         &  & $r\equiv 1$ (mod 3) & $Q\cup \{x_3, x_6,\ldots,x_{r-1}\}$  \\ \cline{3-4}
           &  & $r\equiv 2$ (mod 3) & $Q\cup \{x_2\}\cup \{ x_4, x_7,\ldots,x_{r-1}\}$  \\ \hline      
           \multirow{3}{4em}{$x_0\notin Q$} & \multirow{3}{4em}{$x_r\notin Q$} & $r\equiv 0$ (mod 3) & $Q\cup \{ x_2,x_5,\ldots,x_{r-1}\}$  \\ \cline{3-4}
         &  & $r\equiv 1$ (mod 3) & $Q\cup \{x_3, x_6,\ldots,x_{r-1}\}$  \\ \cline{3-4}
           &  & $r\equiv 2$ (mod 3) & $Q\cup \{ x_1, x_4,\ldots,x_{r-1}\}$  \\ \hline
    \end{tabular}
    \caption{$Q'$ depending on the case and the length of $P_k$}
    \label{tabla quasi}
\end{table} 

It is straightforward that $Q'$ is a quasi-kernel of $D$. Furthermore, for every vertex of $P_k$ that we add to $Q'$, there is at least one vertex of $P_k$ that is not in $Q'$. Therefore, $Q'$ is a small quasi-kernel of $D$.

\end{proof}

\begin{figure}
\begin{center}
\includegraphics[scale=1]{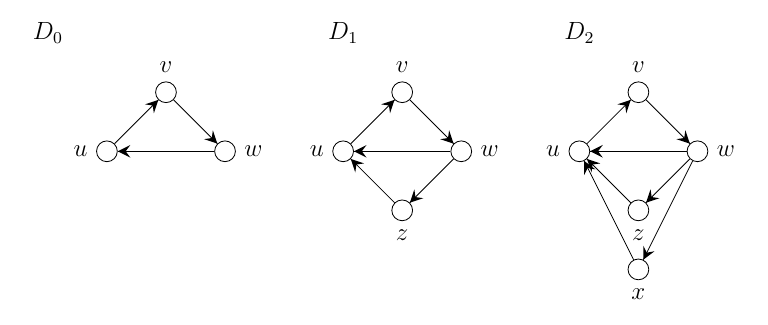}
\end{center}
\caption{Digraph in $\mathcal{LE}_2$}
\label{fig quasi2}
\end{figure}

Now, consider the digraphs in \Cref{fig quasi2}. Note that $Q_1=\{ w\}$ is a small quasi-kernel of $D_0$, however when we consider the ear $P_0=\{w,z,u\}$ of $D_0$, it follows that no vertex of $P_0$ can be added to $Q_1$, such that the resulting set is a quasi-kernel of $D_1$. 
On the other hand, note that $Q_2=\{ v\}$ is a small quasi-kernel of $D_0$, following the idea of the proof of \Cref{thm quasi LE3}, we add $z$ to $Q_2$, obtaining $Q_3=\{v,z\}$ which is a small quasi-kernel of $D_1$. Now, consider the ear $P_1=\{ w,x,u\}$ of $D_1$, again following the idea of the proof of \Cref{thm quasi LE3}, we add $x$ to $Q_3$, obtaining $Q_4=\{v,z,x\}$ which is a quasi-kernel of $D_2$ but is not a small quasi-kernel.
Therefore, when there are ears of length 2, the technique of adding vertices of the last ear to the quasi-kernel obtained in the previous strong subdigraph is not sufficient to guarantee the existence of a small quasi-kernel of the resulting digraph. Therefore, proving that the small quasi-kernel conjecture holds for the families $\mathcal{LE}_1$ and $\mathcal{LE}_2$ remains open.
Similarly, we consider that some other open problems can be solved, at least partially, for some of the $\mathcal{LE}_i$ families.

\section{Kernels in digraphs with long ears}\label{sec kernels}

In this section, some sufficient conditions for a strong nonseparable digraph in $\mathcal{LE}_2$ with a kernel, to imply that the previous subdigraph in the ear decomposition has a kernel, are presented. In addition, some sufficient conditions for a strong nonseparable digraph in $\mathcal{LE}_2$, with a kernel to imply that the following digraph in an ear decomposition has a kernel too, are shown.

\begin{lemma}\label{thm: H' kernel then H kernel}
Let $D$ be a strong nonseparable digraph, let $H$ be a strong nonseparable digraph, and let $P=(x_0,x_1,\ldots,x_{r-1},x_r)$ be an ear of $H$ in $D$, with $l(P)\geq 2$. If $H'=H\cup P$ has a kernel $N'$, and one of the following assertions holds:

\begin{enumerate}
\item $x_0,x_r\in N'$.
\item $x_0\in N'$ and $x_r\notin N'$.
\item $x_0\notin N'$, $x_r\in N'$ and $l(P)$ is even.
\item $x_0,x_r\notin N'$ and $l(P)$ is odd.
\end{enumerate}
 Then $H$ has a kernel.
\end{lemma}

\begin{proof}
Let $D$, $H$, $H'$ and $N'$ as in the hypothesis. By definition, $H$ is an induced subdigraph of $H'$. We will prove that $N=N'\cap V(H)$ is a kernel of $H$. Since $N'$ is independent in $H'$, it follows that $N$ is independent in $H$.
Let $y$ be a vertex in $V(H)\setminus N$. It follows that there is $w\in V(H')$ such that $(y,w)\in A(H')$. If $w\in V(H)$, then $(y,w)\in A(H)$. 
Since $P$ is an ear of $H$ with a length of at least $2$, the only vertex in $H$ whose out-neighborhood intersects $V(P)\setminus V(H)$ is $x_0$. 
 Thus, if $y\neq x_0$, then $w\in V(H)$.
 So, if assertion 1 or 2 is true, then $y \neq x_0$, in this way, assume that $y=x_0$, and assertion 3 or 4 holds.

\emph{Case 1}. $x_0\notin N'$, $x_r\in N'$ and $l(P)$ is even. Since $P$ is a path in $H'$, with even length and $x_r\in N'$, we have that $x_i\notin N'$ for every $i\in \{ 1,3,\ldots,r-1\}$, it implies that $w\neq x_1$. Hence $w\in V(H)$.

\emph{Case 2}. $x_0,x_r\notin N'$ and $l(P)$ is odd. Since $P$ is a path in $H'$, with odd length and $x_r\notin N'$, we have that $x_i\notin N'$ for every $i\in \{ 1,3,\ldots,r-2\}$, it implies that $w\neq x_1$. Hence, $w\in V(H)$.

Therefore, $N$ is an absorbent set in $H$, and a kernel of $H$.
\end{proof}

\begin{corollary}\label{cor: H no kernel and H' kernel then options}
Let $D$ be a strong nonseparable digraph, let $H$ be a strong nonseparable digraph, and let $P=(x_0,x_1,\ldots,x_{r-1},x_r)$ be an ear of $H$ in $D$, with $l(P)\geq 2$. If $H$ has no kernel and $H'=H\cup P$ has a kernel $N'$, then one of the following assertions holds:

\begin{enumerate}
\item $x_0\notin N'$, $x_r\in N'$ and $l(P)$ is odd.
\item $x_0,x_r\notin N'$ and $l(P)$ is even.
\end{enumerate} 
\end{corollary}

Note that if we apply \Cref{thm: H' kernel then H kernel} in an ear decomposition of a digraph of $\mathcal{LE}_2$, we obtain a sufficient condition for a digraph with a kernel, to imply that the previous subdigraph in the ear decomposition has a kernel.

\begin{theorem}\label{thm: D kernel then Dk-1 has kernel}
Let $D$ be a strong nonseparable digraph, let $(D_0,D_1,\ldots,D_k)$ be an ear decomposition of $D$ where $P_{i-1}=(x_0,x_1,\ldots,x_{r-1},x_r)$ is the ear of $D_{i-1}$ in $D$, with $l(P_{i-1})\geq 2$. If $D_i$ has a kernel $N$, and one of the following assertions hold:

\begin{enumerate}
\item $x_0,x_r\in N$.
\item $x_0\in N$ and $x_r\notin N$.
\item $x_0\notin N$, $x_r\in N$ and $l(P_{i-1})$ is even.
\item $x_0,x_r\notin N$ and $l(P_{i-1})$ is odd.
\end{enumerate}
 Then $D_{i-1}$ has a kernel.
\end{theorem}

\begin{corollary}\label{coro: Di kernel Di-1 no kernel thn options}
Let $D$ be a strong nonseparable digraph, let $(D_0,D_1,\ldots,D_k)$ be an ear decomposition of $D$ where $P_{i-1}=(x_0,x_1,\ldots,x_{r-1},x_r)$ is the ear of $D_{i-1}$ in $D$, with $l(P)\geq 2$. If $D_i$ has a kernel $N$ but $D_{i-1}$ has no kernel, then one of the following assertions holds:

\begin{enumerate}
\item $x_0\notin N$, $x_i\in N$ and $l(P_{i-1})$ is odd.
\item $x_0,x_i\notin N$ and $l(P_{i-1})$ is even.
\end{enumerate}
\end{corollary}

From \Cref{thm: D kernel then Dk-1 has kernel} and \Cref{coro: Di kernel Di-1 no kernel thn options}, we have the following result.

\begin{theorem}
Let $D$ be a strong nonseparable digraph, and let $(D_0,D_1,\ldots,D_k)$ be an ear decomposition of $D$ where
every ear has a length of at least 2. If $D$ has a kernel $N$, then one of the following assertions holds:

\begin{enumerate}
\item $D_j$ has a kernel for every $j\in \{0,\ldots,k\}$. In particular, $D_0$ is an even cycle.
\item There is $j\in \{0,\ldots,k\}$ such that $D_j$ has no kernel but $D_i$ has a kernel for every $i\geq j+1$. In particular, if $P_{j}=(x_0,x_1,\ldots,x_{r-1},x_r)$ is the ear of $D_j$ and $N$ is a kernel of $D_{j+1}$, then one of the following assertions holds:
\begin{enumerate}
\item $x_0\notin N$, $x_r\in N$ and $l(P_{j})$ is odd.
\item $x_0,x_r\notin N$ and $l(P_{j})$ is even.
\end{enumerate}
\end{enumerate}

\end{theorem}

On the other hand, we will give conditions to a strong non-separable digraph $H$ has a kernel, then the digraph $H\cup P$ has a kernel, where $P$ is a long ear of $H$.

\begin{lemma}\label{lem: H kernel then H' kernel}
Let $D$ be a strong nonseparable digraph, let $H$ be a strong nonseparable digraph, and let $P=(x_0,y_0,\ldots,x_{r-1},x_r)$ be an ear of $H$ in $D$, with $l(P)\geq 2$. If $H$ has a kernel $N$ and one of the following assertions holds:

\begin{enumerate}
\item $x_0,x_r\in N$ and $l(P)$ is even.
\item $x_0\in N$, $x_r\notin N$ and $l(P)$ is odd.
\item $x_0\notin N$ and $x_r\in N$.
\item $x_0,x_r\notin N$.
\end{enumerate}
 Then $H'=H\cup P$ has a kernel.
 \end{lemma}
 
 \begin{proof}
 Let $D$, $H$, $N$ and $H'$ as in the hypothesis.
Observe that $H$ is an induced subdigraph of $H'$, it follows that $N$ is also an independent set of $H'$.
  We will divide the proof into cases, depending on which of the assertions of the hypothesis hold. For each of the cases, we will give the kernel for $H'$, whose proof follows directly from the definition of ear and the case in which we are dealing.
 
 \emph{Case 1}. $x_0,x_r\in N$ and $l(P)$ is even. The kernel of $H'$ is $N\cup \{ x_2,x_4,\ldots,x_{r-2}\}$.
 
  \emph{Case 2}. $x_0\in N$, $x_r\notin N$ and $l(P)$ is odd. The kernel of $H'$ is $N\cup \{ x_2,x_4,\ldots,x_{r-1}\}$.
  
  \emph{Case 3}. $x_0\notin N$ and $x_r\in N$. 
  The kernel of $H'$ is $N\cup \{ x_2,x_4,\ldots,x_{r-2}\}$ when $l(P)$ is even; otherwise the kernel of $H'$ is $N\cup \{ x_1,x_3,\ldots,x_{r-2}\}$.
  
  \emph{Case 4}. $x_0,x_r\notin N$. 
  The kernel of $H'$ is $N\cup \{ x_1,x_3,\ldots,x_{r-1}\}$ when $l(P)$ is even; otherwise the kernel of $H'$ is $N\cup \{ x_2,x_4,\ldots,x_{r-1}\}$.
 \end{proof}

\begin{corollary}\label{coro: H kernel and H' no kernel then options}
Let $D$ be a strong nonseparable digraph, let $H$ be a strong nonseparable digraph, and let $P=(x_0,x_1,\ldots,x_{r-1},x_r)$ be an ear of $H$ in $D$, with $l(P)\geq 2$. If $H$ has a kernel $N$ and $H'=H\cup P$ has no kernel, then one of the following assertions holds:

\begin{enumerate}
\item $x_0,x_r\in N$ and $l(P)$ is odd.
\item $x_0\in N$, $x_r\notin N$ and $l(P)$ is even.

\end{enumerate} 
\end{corollary}

Similarly, note that if we apply \Cref{lem: H kernel then H' kernel} on an ear decomposition of a digraph of $\mathcal{LE}_2$, we obtain a sufficient condition for a subdigraph of the ear decomposition with a kernel to imply that the following subdigraph in the ear decomposition has a kernel too.

\begin{theorem}\label{thm: Di kernel + options thn Di+1 kernel}
Let $D$ be a strong nonseparable digraph, let $(D_0,D_1,\ldots,D_k)$ be an ear decomposition of $D$ where $P_{i}=(x_0,x_1,\ldots,x_{r-1},x_r)$ is the ear of $D_{i}$ in $D$, with $l(P_{i})\geq 2$.
 If $D_{i}$ has a kernel $N$, and one of the following assertions holds:

\begin{enumerate}
\item $x_0,x_r\in N$ and $l(P_i)$ is even.
\item $x_0\in N$, $x_r\notin N$ and $l(P_i)$ is odd.
\item $x_0\notin N$ and $x_r\in N$.
\item $x_0,x_r\notin N$.
\end{enumerate}
 Then $D_{i+1}$ has a kernel.

 \end{theorem}

\begin{corollary}\label{coro: Di kernel + Di+1 no kernel thn options}
Let $D$ be a strong nonseparable digraph, let $(D_0,D_1,\ldots,D_k)$ be an ear decomposition of $D$ where $P_{i}=(x_0,x_1,\ldots,x_{r-1},x_r)$ is the ear of $D_{i}$ in $D$, with $l(P_{i})\geq 2$.
 If $D_{i}$ has a kernel $N$  but $D_{i+1}$ has no kernel, then one of the following assertions holds:

\begin{enumerate}
\item $x_0,x_r\in N$ and $l(P_i)$ is odd.
\item $x_0\in N$, $x_r\notin N$ and $l(P_i)$ is even.

\end{enumerate} 
\end{corollary}

From \Cref{thm: Di kernel + options thn Di+1 kernel} and \Cref{coro: Di kernel + Di+1 no kernel thn options}, we have the following result.

\begin{theorem}
Let $D$ be a strong nonseparable digraph, and let $(D_0,D_1,\ldots,D_k)$ be an ear decomposition of $D$ where
every ear has a length of at least 2. If $D$ has no kernel, then one of the following assertions holds:

\begin{enumerate}
\item $D_j$ has no kernel for every $j\in \{0,\ldots,k\}$. In particular, $D_0$ is an odd cycle.
\item There is $j\in \{0,\ldots,k\}$ such that $D_j$ has a kernel $N$ but $D_i$ has no kernel for every $i\geq j+1$. In particular, if $P_{j}=(x_0,x_1,\ldots,x_{r-1},x_r)$ is the ear of $D_j$, then one of the following assertions holds:
\begin{enumerate}
\item $x_0,x_r\in N$ and $l(P_j)$ is odd.
\item $x_0\in N$, $x_r\notin N$ and $l(P_j)$ is even.
\end{enumerate} 
\end{enumerate}

\end{theorem}

\section{Classical colorings}\label{sec classical colorings}

In this section, we will prove that digraphs in $\mathcal{LE}_2$ have a chromatic number at most 3, and as a consequence a digraph in $\mathcal{LE}_2$ has a dichromatic number 2 or 3.

\begin{theorem}

Let $D$ be a strong digraph. If $D \in \mathcal{LE}_2$,
 then $\chi(D)\leq 3$.
\end{theorem}

\begin{proof}
Let $D$ be a strong digraph with an ear decomposition $(D_0,D_1,\ldots,D_k)$ such that every ear of $D_i$ has a length of at least two. 
We will proceed by induction over $k$. 
If $k=0$, then $D$ is a cycle, and there is $c:V(D_0)\rightarrow \{1,2,3\}$ a proper 3-coloring of $V(D_0)$.
 Suppose that if $D'$ is a strong digraph with an ear decomposition $(D'_0,D'_1,\ldots,D'_k)$ such that every ear of $D'_i$ has a length of at least two, then there is $c':V(D')\rightarrow \{1,2,3\}$ a proper $3$-coloring of $D'$.
Suppose that $D$ is a strong digraph with an ear decomposition $(D_0,D_1,\ldots,D_k,D_{k+1})$ such that every ear of $D_i$ has a length of at least two. Let $P_{k}=(x_0,x_1,\ldots,x_r)$ be the ear of $D_k$ in $D$.
 Note that $D'=D_k$ is a strong digraph with an ear decomposition $(D_0,D_1,\ldots,D_k,)$ such that every ear of $D_i$ has a length of at least two, by induction hypothesis, $c':V(D')\rightarrow \{1,2,3\}$ a proper $3$-coloring of $D'$. We will consider 3 cases.
 
 \emph{Case 1}. $l(P_k)=2$, it follows that $x_r=x_2$. Hence, there is a color in $\{1,2,3\}$, which is not used to color either $x_0$ or $x_2$. Color $x_1$ with that color and obtain a proper $3$-coloring of $D$.
 
 \emph{Case 2}. $l(P_k)=3$, it follows that $x_r=x_3$. We have to color $x_1$ and $x_2$.
Note that there is a color in $\{1,2,3\}$, which is different than $c'(x_0)$ and $c'(x_3)$; color $x_1$ with that color, say $i$. Now, there is a color in $\{1,2,3\}$, which is different than $i$ and $c'(x_3)$, color $x_2$ with that color, and obtain a proper $3$-coloring of $D$.

\emph{Case 3}. $l(P_k)>3$.
There is a color in $\{1,2,3\}$, which is different than $c'(x_0)$ and $c'(x_r)$; color $x_1$ and $x_{r-1}$ with that color, say $i$. Observe that $(x_2,x_3,\ldots,x_{r-2})$ is a bipartite digraph, hence, we can color its vertices with colors other than $i$, and obtain a proper $3$-coloring of $D$.

\end{proof}

The \emph{dichromatic number}  $\overrightarrow{\chi}(D)$, of a digraph $D$ is the minimum integer $k$ such that $D$ admits
a $k$-coloring where the chromatic classes induce acyclic subdigraphs of $D$. 
Since every chromatic class of a proper $k$-coloring of a digraph $D$ induces an acyclic subdigraph of $D$, we have that $\overrightarrow{\chi}(D)\leq \chi(D)$. Hence, we have the following corollary.

\begin{corollary}
Let $D$ be a strong digraph. If $D\in \mathcal{LE}_2$,
 then $2\leq \overrightarrow{\chi}(D)\leq 3$.
\end{corollary}

\section{Oriented coloring}\label{sec oriented coloring}

In this section, we bound the oriented chromatic number of asymmetrical digraphs in $\mathcal{LE}_3$. Furthermore, we will construct a family of digraphs with arbitrarily high oriented chromatic number where every ear has a length of 2, Hence, the hypothesis that the ears have a length of at least 3 is tight.

A digraph $D$ with, order $n$, is \emph{vertex-pancyclic} if and only if every vertex is contained in a $k$-cycle of $D$, for every $k\in\{3,\ldots,n\}$. In \cite{Moon1966OnSO}, Moon proved that every strong tournament is vertex-pancyclic.

\begin{proposition}
There is only one tournament with order $6$, such that for every $i,j\in \{0,1,2,3,4,5\}$ there are a walk from $i$ to $j$ and a walk from $k$ to $i$, both with length $k$, for each $k\in\{3,4,5\}$.
\end{proposition}

According to the Combinatorial Data website \cite{combinatorial}, there are 56 non-isomorphic tournaments of order 6, each is given as the upper triangle of the adjacency matrix in row order, on one line without spaces. For every $i$ in the interval $[0,47]\cup [49,55]$, \Cref{table: T fails} shows two vertices $i$ and $j$ such that there is no walk from $r$ to $j$ of length $3$ in $T_i$.
 On the other hand, \Cref{table: T48 walks} shows a walk from $i$ to $j$ of length $k$ in $T_{48}$, for each $k\in\{3,4,5\}$.
 
 \begin{table}[!ht]
    \centering
    \begin{tabular}{|l|c|c|c|}
    \hline
        $T_i$ with $i$ in the interval & initial vertex & final vertex & length  \\ \hline
        $[0,11]\cup [27,45]\cup [47]\cup [49]$ & 0 & 1 & 3  \\ \hline
        $[12,13]$ & 0 & 5 & 3  \\ \hline
        $[14]\cup [19,20]\cup [50]$ & 0 & 3 & 3  \\ \hline
        $[15,18]$ & 0 & 4 & 3  \\ \hline
        $[21,26]\cup [52]\cup [54,55]$ & 0 & 2 & 3  \\ \hline
        $[46]$ & 1 & 5 & 3  \\ \hline
        $[51]$ & 3 & 4 & 3  \\ \hline
        $[53]$ & 2 & 4 & 3  \\ \hline
    \end{tabular}
    \caption{There is no 3-path from the initial vertex to the final vertex in $T_i$.}
    \label{table: T fails}
\end{table}

\begin{table}[!ht]
    \centering
    \begin{tabular}{|c|c|c|c|c|}
    \hline
\multicolumn{2}{|c|}{ } & \multicolumn{3}{|c|}{length}\\ \hline
        initial vertex & final vertex & 3 & 4 & 5 \\ \hline
        0 & 1 & (0, 3, 4, 1) &  (0, 3, 4, 0, 1) &  (0, 3, 4, 5, 3, 1) \\ \hline
        ~ & 2 & (0, 1, 5, 2) & (0, 3, 1, 5, 2) & (0, 3, 1, 5, 3, 2) \\ \hline
        ~ & 3 &  (0, 1, 5, 3) & (0, 1, 5, 0, 3) & (0, 1, 5, 2, 0, 3) \\ \hline
        ~ & 4 &  (0, 3, 2, 4) &  (0, 1, 5, 3, 4) & (0, 3, 1, 5, 2, 4) \\ \hline
        ~ & 5 &  (0, 3, 1, 5) &  (0, 1, 2, 4, 5) &  (0, 1, 2, 0, 1, 5) \\ \hline
        1 & 0 & (1, 2, 4, 0) &  (1, 2, 4, 5, 0) &  (1, 2, 4, 1, 2, 0) \\ \hline
        ~ & 2 & (1, 5, 3, 2) &  (1, 5, 0, 3, 2) & (1, 5, 3, 1, 5, 2) \\ \hline
        ~ & 3 &  (1, 2, 0, 3) & (1, 2, 4, 0, 3) &  (1, 2, 0, 1, 5, 3) \\ \hline
        ~ & 4 & (1, 5, 2, 4) & (1, 2, 0, 3, 4) &  (1, 2, 0, 1, 2, 4) \\ \hline
        ~ & 5 &  (1, 2, 4, 5) & (1, 2, 0, 1, 5) &  (1, 2, 0, 3, 4, 5) \\ \hline
        2 & 0 &  (2, 4, 5, 0) & (2, 4, 1, 2, 0) & (2, 4, 1, 2, 4, 0) \\ \hline
        ~ & 1 & (2, 0, 3, 1) &  (2, 0, 3, 4, 1) &  (2, 0, 3, 2, 0, 1) \\ \hline
        ~ & 3 &  (2, 4, 0, 3) & (2, 0, 1, 5, 3) &  (2, 0, 1, 2, 0, 3) \\ \hline
        ~ & 4 &  (2, 0, 3, 4) & (2, 0, 3, 2, 4) & (2, 0, 1, 5, 2, 4) \\ \hline
        ~ & 5 &  (2, 0, 1, 5) &  (2, 0, 3, 1, 5) & (2, 0, 1, 2, 4, 5) \\ \hline
        3 & 0 & (3, 1, 5, 0) &  (3, 1, 2, 4, 0) &  (3, 1, 2, 4, 5, 0) \\ \hline
        ~ & 1 & (3, 2, 0, 1) & (3, 2, 4, 0, 1) & (3, 2, 0, 3, 4, 1) \\ \hline
        ~ & 2 & (3, 1, 5, 2) &  (3, 1, 5, 3, 2) & (3, 1, 5, 0, 1, 2) \\ \hline
        ~ & 4 & (3, 1, 2, 4) &  (3, 1, 5, 2, 4) &  (3, 1, 2, 0, 3, 4) \\ \hline
        ~ & 5 & (3, 2, 4, 5) &  (3, 1, 2, 4, 5) &  (3, 1, 2, 0, 1, 5) \\ \hline
        4 & 0 &  (4, 1, 5, 0) &  (4, 1, 2, 4, 0) & (4, 1, 2, 4, 5, 0) \\ \hline
        ~ & 1 & (4, 0, 3, 1) &  (4, 0, 3, 4, 1) & (4, 0, 3, 2, 0, 1) \\ \hline
        ~ & 2 &  (4, 0, 1, 2) & (4, 0, 1, 5, 2) &  (4, 1, 5, 0, 1, 2) \\ \hline
        ~ & 3 &  (4, 1, 5, 3) &  (4, 0, 1, 5, 3) &  (4, 0, 1, 2, 0, 3) \\ \hline
        ~ & 5 &  (4, 0, 1, 5) &  (4, 0, 3, 1, 5) &  (4, 0, 1, 2, 4, 5) \\ \hline
        5 & 0 & (5, 2, 4, 0) & (5, 2, 4, 5, 0) & (5, 2, 4, 1, 2, 0) \\ \hline
        ~ & 1 & (5, 2, 4, 1) & (5, 3, 2, 0, 1) &  (5, 3, 4, 5, 0, 1) \\ \hline
        ~ & 2 &  (5, 0, 1, 2) &  (5, 0, 1, 5, 2) &  (5, 0, 1, 5, 3, 2) \\ \hline
        ~ & 3 &  (5, 2, 0, 3) &  (5, 0, 1, 5, 3) &  (5, 0, 1, 2, 0, 3) \\ \hline
        ~ & 4 &  (5, 0, 3, 4) &  (5, 0, 1, 2, 4) &  (5, 0, 1, 5, 2, 4) \\ \hline
    \end{tabular}
    \caption{Walks of lengths 3, 4, and 5, between two vertices in $T_{48}$.}
    \label{table: T48 walks}
\end{table}

The tournament $T_{48}$ is depicted in \Cref{fig T48};
from now on, we will refer to $T_{48}$ simply as $T$.

\begin{figure}
\begin{center}
\includegraphics[scale=1]{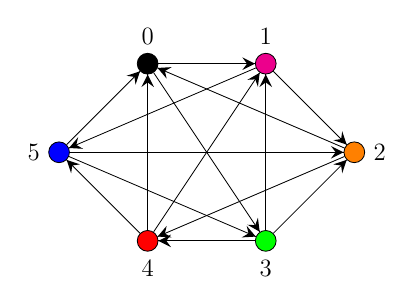}
\end{center}
\caption{Tournament $T_{48}$ with 6 vertices}
\label{fig T48}
\end{figure}

\begin{lemma}\label{lem: homomorphism Cn to T}
If $\overrightarrow{C}_n$ is a cycle with $n\geq 3$, then there is a  homomorphism of $\overrightarrow{C}_n$ to $T$.
\end{lemma}

\begin{proof}
Suppose that $\overrightarrow{C}_n=(x_0,x_1,\ldots,x_{n-1},x_0)$.
Note that $T$ is a strong tournament, by Moon's theorem,
we have that every $i\in V(T)$ is contained in a $k$-cycle, for every $k\in \{3,4,5,6\}$, we denote that cycle by $\gamma^k_{i}$.

It is clear that if $n\in \{3,4,5,6\}$, then there is an homomorphism from $\overrightarrow{C}_n$ to $\gamma^n_1$. Assume that $n\geq 7$. We consider $3$ cases.

\emph{Case 1}. $n\equiv 0$ (mod $3$). We define $\phi:V(\overrightarrow{C}_n)\rightarrow V(T)$ such that $\phi(x_i)=r$, where $r\in \{0,1,2\}$ and $i\cong r$ (mod $3$). Since $(0,1,2,0)$ is a cycle of $T$, we have that if $(x_i,x_{i+1})\in A(\overrightarrow{C}_n)$, then $(\phi(x_i),\phi(x_{i+1}))=(r,r+1)\in A(T)$.

\emph{Case 2}. $n\equiv 1$ (mod $3$). We have that $n=3k+1$ define $\phi:V(\overrightarrow{C}_n)\rightarrow V(T)$ such that

$$\phi (x_i) =
\left \{ \begin{array}{ll}
 r & \text{ if } i\in \{0,\ldots,3k \}, r\in \{0,1,2\} \text{ and } i\equiv r (\text{ mod } 3)\\
4 & \text{ if }i=n.
\end{array}
\right .$$

 Since $(0,1,2,0)$ and $(0,1,2,4,0)$ are cycles of $T$, we have that if $(x_i,x_{i+1})\in A(\overrightarrow{C}_n)$, then $(\phi(x_i),\phi(x_{i+1}))=(r,r+1)\in A(T)$.
 
 \emph{Case 3}. $n\equiv 2$ (mod $3$). We have that $n=3k+2$ define $\phi:V(\overrightarrow{C}_n)\rightarrow V(T)$ such that

$$\phi (x_i) =
\left \{ \begin{array}{ll}
 r & \text{ if } i\in \{0,\ldots,3k \}, r\in \{0,1,2\} \text{ and } i\equiv r (\text{ mod } 3)\\
4 & \text{ if }i=3k+1\\
5 & \text{ if }i=3k+2.\\
\end{array}
\right .$$

 Since $(0,1,2,0)$ and $(0,1,2,4,5,0)$ are cycles of $T$, we have that if $(x_i,x_{i+1})\in A(\overrightarrow{C}_n)$, then $(\phi(x_i),\phi(x_{i+1}))=(r,r+1)\in A(T)$.

Therefore, there is a homomorphism from $\overrightarrow{C}_n$ to $T$.
\end{proof}

\begin{theorem}\label{thm: xo D is at most 6}
Let $D$ be a strong asymmetrical digraph. If $D\in \mathcal{LE}_3$, then $\chi_o(D)\leq 6$.
\end{theorem}

\begin{proof}

Note that $T$ is a strong tournament, by
Moon's theorem,  we have that every $i\in V(T)$ is contained in a $k$-cycle, for every $k\in \{3,4,5,6\}$, we denote that cycle by $\gamma^k_{i}$.
 If $i$ and $j$ are two different vertices of $V(T)$, then there is a $k$-path from $i$ to $j$ $k\in \{3,4,5\}$ (see \Cref{table: T48 walks}), we denote that path by $P^k_{ij}$.

Let $D$ be a digraph with an ear decomposition $(D_0,D_1,\ldots,D_k)$ such that every ear of $D_i$ has a length of at least three.
We will prove that there is a homomorphism from $D$ to $T$. We will proceed by induction over $k$. 
If $k=0$, then $D$ is a cycle, by \Cref{lem: homomorphism Cn to T} there is a homomorphism from $D$ to $T$.
 Suppose that if $D'$ is a strong digraph with an ear decomposition $(D'_0,D'_1,\ldots,D'_k)$ such that every ear of $D'_i$ has a length of at least three, then there is $\phi':V(D')\rightarrow V(T)$ a homomorphism from $D'$ to $T$.
Suppose that $D$ is a strong digraph with an ear decomposition $(D_0,D_1,\ldots,D_k,D_{k+1})$ such that every ear of $D_i$ has a length of at least three. Let $P_{k}=(x_0,x_1,\ldots,x_r)$ be the ear of $D_k$ in $D$.
 Note that $D'=D_k$ is a strong digraph with an ear decomposition $(D_0,D_1,\ldots,D_k)$ such that every ear of $D_i$ has length of at least three, by induction hypothesis, there is $\phi':V(D')\rightarrow V(T)$ an homomorphism from $D'$ to $T$. We define $\phi:V(D)\rightarrow V(T)$ such that $\phi(x)=\phi'(x)$ if $x\in V(D')$ and, depending the length of $P_k$ and if $P_k$ is a cycle or not, we consider the following cases.
 
\emph{Case 1}. $x_0=x_r$ and $r\equiv 0$ (mod 3). It follows that $r=3t$. Suppose that $\phi (x_0)=i$ with $i\in V(T)$, and $\gamma^3_i=(i=i_0,i_1,i_2,i_0)$ a 3-cycle in $T$. 
For every $j\in \{ 0,1,\ldots,r-1\}$, we define $\phi(x_j)=i_s$, where $j\cong s$ (mod $3$).
Since $\gamma^3_i$ is a cycle of $T$, we have that if $(x_i,x_{i+1})\in A(P_k)$, then $(\phi(x_i),\phi(x_{i+1}))=(i_s,i_{s+1})\in A(T)$.

\emph{Case 2}. $x_0\neq x_r$ and $r\equiv 0$ (mod 3). It follows that $r=3t$. Suppose that $\phi (x_0)=i$ and $\phi(x_r)=j$ with $i,j\in V(T)$. Consider $\gamma^3_i=(i=i_0,i_1,i_2,i_0)$ a 3-cycle in $T$, and $P^3_{ij}=(i=j_0,j_1,j_2,j_3=j)$ a 3-path from $i$ to $j$ in $T$.

If $i=j$, then, for every $m\in \{ 0,1,\ldots,3t\}$, we define $\phi(x_m)=i_s$, where $m\cong s$ (mod $3$). Otherwise, $i\neq j$, for every $m\in \{ 0,1,\ldots,3(t-1)\}$, we define $\phi(x_m)=i_s$, where $m\cong s$ (mod $3$), and for every $m\in \{ 3t-2,3t-1,3t\}$, $\phi(x_m)=j_s$, where $m\cong s$ (mod $3$).

Since $\gamma^3_i$ is a cycle and $P^3_{ij}$ is a path from $i$ to $j$ in $T$, we have that if $(x_i,x_{i+1})\in A(P_k)$, then $(\phi(x_i),\phi(x_{i+1}))=(i_s,i_{s+1})\in A(T)$.

\emph{Case 3}. $x_0=x_r$ and $r\equiv 1$ (mod 3). It follows that $r=3t+1$. Suppose that $\phi (x_0)=i$ with $i\in V(T)$. Consider $\gamma^3_i=(i=i_0,i_1,i_2,i_0)$ and $\gamma^4_i=(i=j_0,j_1,j_2,j_3,j_0)$ a 3-cycle and a 4-cycle, respectively, in $T$. 

For every $m\in \{ 0,1,\ldots,3(t-1)\}$, we define $\phi(x_m)=i_s$, where $m\cong s$ (mod $3$), and for every $m\in \{ 3t-2,3t-1,3t,3t+1\}$, $\phi(x_m)=j_s$, where $m\cong s$ (mod $4$).

Since $\gamma^3_i$ and $\gamma^4_i$ are two cycles of $T$, that start in $i$, we have that if $(x_i,x_{i+1})\in A(P_k)$, then $(\phi(x_i),\phi(x_{i+1}))=(i_s,i_{s+1})\in A(T)$.

\emph{Case 4}. $x_0\neq x_r$ and $r\equiv 1$ (mod 3). It follows that $r=3t+1$. Suppose that $\phi (x_0)=i$ and $\phi(x_r)=j$ with $i,j\in V(T)$. 
Consider $\gamma^3_i=(i=i_0,i_1,i_2,i_0)$ a 3-cycle in $T$, $\gamma^4_i=(i=q_0,q_1,q_2,q_3,q_0)$ a 4-cycle in $T$ and $P^4_{ij}=(i=j_0,j_1,j_2,j_3,j_4=j)$ a 4-path from $i$ to $j$ in $T$.

For every $m\in \{ 0,1,\ldots,3(t-1)\}$, we define $\phi(x_m)=i_s$, where $m\cong s$ (mod $3$).
If $i=j$, then for every $m\in \{ 3t-2,3t-1,3t,3t+1\}$, $\phi(x_m)=q_s$ where $m\cong s$ (mod $4$). Otherwise, $i\neq j$, for every $m\in \{ 3t-2,3t-1,3t,3t+1\}$, $\phi(x_m)=j_s$, where $m\cong s$ (mod $4$).

Since $\gamma^3_i$ and $\gamma^4_i$ are two cycles of $T$, that start in $i$, and $P^4_{ij}$ is a $4$-path from $i$ to $j$ in $T$, we have that if $(x_i,x_{i+1})\in A(P_k)$, then $(\phi(x_i),\phi(x_{i+1}))=(i_s,i_{s+1})\in A(T)$.
 
 \emph{Case 5}. $x_0=x_r$ and $r\equiv 2$ (mod 3). It follows that $r=3t+2$. Suppose that $\phi (x_0)=i$ with $i\in V(T)$. Consider $\gamma^3_i=(i=i_0,i_1,i_2,i_0)$ and $\gamma^5_i=(i=j_0,j_1,j_2,j_3,j_4,j_0)$ a 3-cycle and a 5-cycle, respectively, in $T$. 

For every $m\in \{ 0,1,\ldots,3(t-1)\}$, we define $\phi(x_m)=i_s$, where $m\cong s$ (mod $3$), and for every $m\in \{ 3t-2,3t-1,3t,3t+1,3t+2\}$, $\phi(x_m)=j_s$, where $m\cong s$ (mod $5$).

Since $\gamma^3_i$ and $\gamma^5_i$ are two cycles of $T$, that start in $i$, we have that if $(x_i,x_{i+1})\in A(P_k)$, then $(\phi(x_i),\phi(x_{i+1}))=(i_s,i_{s+1})\in A(T)$.

\emph{Case 6}. $x_0\neq x_r$ and $r\equiv 2$ (mod 3). It follows that $r=3t+2$. Suppose that $\phi (x_0)=i$ and $\phi(x_r)=j$ with $i,j\in V(T)$. 
Consider $\gamma^3_i=(i=i_0,i_1,i_2,i_0)$ a 3-cycle in $T$, $\gamma^5_i=(i=q_0,q_1,q_2,q_3,q_4,q_0)$ a 5-cycle in $T$ and $P^5_{ij}=(i=j_0,j_1,j_2,j_3,j_4,j_5=j)$ a 5-path from $i$ to $j$ in $T$.

For every $m\in \{ 0,1,\ldots,3(t-1)\}$, we define $\phi(x_m)=i_s$, where $m\cong s$ (mod $3$).
If $i=j$, then for every $m\in \{ 3t-2,3t-1,3t,3t+1,3t+2\}$, $\phi(x_m)=q_s$ where $m\cong s$ (mod $5$). Otherwise, $i\neq j$, for every $m\in \{ 3t-2,3t-1,3t,3t+1,3t+2\}$, $\phi(x_m)=j_s$, where $m\cong s$ (mod $5$).

Since $\gamma^3_i$ and $\gamma^5_i$ are two cycles of $T$, that starts in $i$, and $P^5_{ij}$ is a $5$-path from $i$ to $j$ in $T$, we have that if $(x_i,x_{i+1})\in A(P_k)$, then $(\phi(x_i),\phi(x_{i+1}))=(i_s,i_{s+1})\in A(T)$.

Hence, $\phi$ is a homomorphism from $D$ to $T$. Therefore, $\chi_o(D)\leq 6$.

\end{proof}

\begin{figure}
\begin{center}
\includegraphics[scale=1]{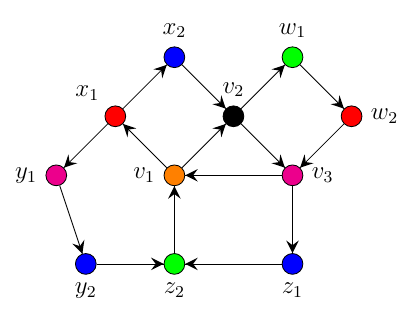}
\end{center}
\caption{Digraph $D$ with an ear decomposition such that every ear has a length 3, and $\chi_o(D)=6$}
\label{fig ears long 3}
\end{figure}

Consider the digraph $D$ in \Cref{fig ears long 3}, and the following subdigraphs of $D$; $D_0=D[\{v_1,v_2,v_3\}]$, $D_1=D[\{v_1,v_2,v_3,x_1,x_2\}]$, $D_2=D[\{v_1,v_2,v_3,x_1,x_2,w_1,w_2\}]$, $D_3=D[\{v_1,v_2,v_3,x_1,x_2,w_1,w_2,z_1,z_2\}]$, and  $D_4=D$. It follows that, \linebreak $(D_0,D_1,D_2,D_3,D_4)$ is an ear decomposition where every ear has length 3.
By \Cref{thm: xo D is at most 6}, we have that $\chi_o(D)\leq 6$, on the other hand, it is not difficult to prove that $\chi_o(D)>5$. Hence, $\chi_o(D)= 6$. We can conclude that the bound in \Cref{thm: xo D is at most 6} is tight.

Let $G_1$ be a digraph such that $G_1\cong \overrightarrow{C}_3$.
For every $i\geq 2$, we construct the digraph $G_i$, obtained from $G_{i-1}$ as follows:
\begin{enumerate}
\item For every pair of distinct vertices $u$ and $v$, add a new vertex $x_{uv}$.
\item Add the arcs $(u,x_{uv})$ and $(x_{uv},v)$.
\end{enumerate}

Observe that every $G_i$ has an ear decomposition, such that the first subdigraph is $G_1$, and every ear has length 2. On the other hand, consider the digraph $G_i$ and $\phi$ an oriented coloring of $G_i$ with $\chi_o(G_i)$ colors. Let $u,v\in V(G_{i-1})$, it follows that there is a vertex $x_{uv}$ of $G_i$ such that $(u,x_{uv},v)$ is a path in $G_i$. It implies that $\phi(u)\neq \phi(v)$. Hence, $|V(G_{i-1})|\leq \chi_o(G_i)$. Therefore, we can conclude that the oriented chromatic number of the family  $\{ G_i\}^{\infty}_{i=0}$ is not bounded.
 Since $\{ G_i\}^{\infty}_{i=0}\subseteq \mathcal{LE}_2$, we conclude that the oriented chromatic number of $\mathcal{LE}_2$ is not bounded.
Even more, by construction, $|A(G_i)|=2|V(G_i)|-3$, which implies that in general, the number of arcs does not need to be dense in digraphs, so that the oriented chromatic number of some digraph families is not bounded.

\begin{figure}
\begin{center}
\includegraphics[width=\textwidth]{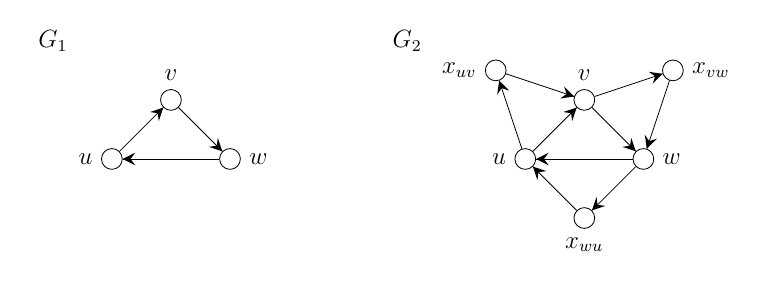}
\end{center}
\caption{Digraphs with $G_1$ and $G_2$}
\label{fig: G1 and G22}
\end{figure}

To conclude,
we propose the study of different theoretical problems, as well as their applications, which are often difficult for digraphs in general, but in these families, an approach to the solution of the problem could be provided.
In particular, we believe that different types of coloring could be addressed in this family, such as acyclic colorings, star colorings, and distance 2-colorings.

\bibliographystyle{plain}
\bibliography{biblio}

\end{document}